\documentclass{article}

\usepackage{amsmath,amsfonts,amssymb,amsthm}
\usepackage{graphicx,subfig}
\usepackage{tabularx}
\usepackage[utf8]{inputenc}
\usepackage{authblk}
\usepackage{xcolor}
\usepackage{verbatim}
\usepackage[title]{appendix}
\usepackage{paralist}
\usepackage{url}
\usepackage{mathtools}
\newtheorem{theorem}{Theorem}
\newtheorem{lemma}{Lemma}

\theoremstyle{definition}
\newtheorem{remark}{Remark}

\usepackage{enumerate}

\makeatletter
\def\blfootnote{\gdef\@thefnmark{}\@footnotetext}
\makeatother

\renewcommand{\le}{\leqslant}
\renewcommand{\ge}{\geqslant}
\renewcommand{\leq}{\leqslant}
\renewcommand{\geq}{\geqslant}
\renewcommand{\emptyset}{\varnothing}

\newcommand{\real}{\mathbb{R}}

\newcommand{\natu}{\mathbb{N}}

\newcommand{\bsx}{\boldsymbol{x}}

\newcommand{\bsone}{\boldsymbol{1}}

\newcommand{\bsgamma}{\boldsymbol{\gamma}}
\newcommand{\bsnu}{\boldsymbol{\nu}}
\newcommand{\bsbeta}{\boldsymbol{\beta}}

\newcommand{\rd}{\,\mathrm{d}}

\newcommand{\e}{\mathbb{E}}

\begin{document}
\title{Sharp convergence bounds for sums of \\  POD and SPOD weights}
\date{}
\author{Zexin Pan\footnote{Institute of Fundamental and Transdisciplinary Research, Zhejiang University, 866 Yuhangtang Road, Xihu District, Hangzhou, Zhejiang Province, 310058, China. (\texttt{zep002@zju.edu.cn}).}}
\maketitle

\begin{abstract}
This work analyzes the convergence of sums of the form $S_{\boldsymbol{\gamma}}(m)=\sum_{v\subseteq \mathbb{N}}\gamma_v m^{|v|}$ with product and order dependent (POD) weights $\gamma_v$. We establish that for a nonnegative sequence $\{\Upsilon_j\mid j\in \mathbb{N}\}$,
$$\sum_{v\subseteq \mathbb{N}} |v|! m^{|v|}\prod_{j\in v} \Upsilon_j<\infty \text{ for all } m>0 \text{ if and only if } \sum_{j=1}^\infty \Upsilon_j<\infty.$$
We further characterize the growth of $S_{\boldsymbol{\gamma}}(m)$ when $\gamma_v=(|v|!)^{\sigma}\prod_{j\in v}j^{-\rho}$ and prove that $\log S_{\boldsymbol{\gamma}}(m)$ is of asymptotic order $m^{1/(\rho-\sigma)}$ when $\rho>\sigma\geq 0$.
We subsequently generalize both the convergence criterion and the asymptotic order of $\log S_{\boldsymbol{\gamma}}(m)$ to smoothness-driven product and order dependent (SPOD) weights, while noting that a full necessary-and-sufficient analogue remains open. Finally, we apply our theory to quasi-Monte Carlo (QMC) integration, showing that interlaced polynomial lattice rules achieve a dimension-independent convergence rate without a commonly imposed assumption in the QMC literature.
\end{abstract}

\section{Introduction}

This paper is motivated by tractability analysis in quasi-Monte Carlo (QMC) methods, where one encounters sums of the form
\begin{equation*}
 S_{\bsgamma}(m)=\sum_{v\subseteq \natu}\gamma_v m^{|v|} ,   
\end{equation*}
with $\natu=\{1,2,3,\dots\}$  and $\bsgamma=\{\gamma_v\mid v\subseteq \natu\}$ a collection of nonnegative real numbers called weights. Here, the summation over $v\subseteq \natu$ runs over all finite subsets of $\natu$.

A central problem, arising for instance in the analysis of lattice rules for integration in weighted Korobov spaces \cite{dick2022lattice}, is to establish under suitable conditions on $\bsgamma$ that
\begin{equation}\label{eqn:summable}
S_{\bsgamma}(m)<\infty \text{ for all } m> 0.    
\end{equation}
Proving \eqref{eqn:summable} is key to demonstrating that certain QMC rules achieve error bounds independent of the problem dimension—a property known as strong tractability in information-based complexity \cite{nova:wozn:2008}. 

Beyond mere convergence, some applications require precise control over the growth of $S_{\bsgamma}(m)$ as $m\to\infty$. In the analysis of base-$b$ digital nets \cite{dick:pill:2010}, for instance, $m$ presents the base-$b$ logarithm of the number of quadrature points $N=b^m$. The goal is often to show $S_{\bsgamma}(m)$ is sub-exponential in $m$:
\begin{equation}\label{eqn:subexp}
  \limsup_{m\to \infty} \frac{1}{m}\log S_{\bsgamma}(m)=0,  
\end{equation}
or equivalently, that for every $\delta>0$ there exists a constant $C_{\bsgamma,\delta}<\infty$ such that $S_{\bsgamma}(m)\leq C_{\bsgamma,\delta} b^{\delta m}$. 
Establishing \eqref{eqn:subexp} for specific weights, such as those in \cite[equation (13)]{c812a4a5-75ac-3614-936d-5782e44941d0}, can imply strong tractability of classical sequences like those of Niederreiter in weighted Sobolev spaces.

The convergence properties of $S_{\bsgamma}(m)$ are straightforward for product weights $\gamma_v=\prod_{j\in v}\gamma_j$.  In this case, the sum factors into an infinite product:
$$S_{\bsgamma}(m)=\prod_{j=1 }^\infty (1+\gamma_j m),$$
from which it follows that both \eqref{eqn:summable} and \eqref{eqn:subexp} are satisfied if $\sum_{j=1}^\infty \gamma_j<\infty$ \cite[Lemma 3]{HICKERNELL2003286}.

The analysis becomes more involved for product and order-dependent (POD) weights. These take the general form
\begin{equation}\label{eqn:PODweight}
 \gamma_v=\Gamma_{|v|}\prod_{j\in v}\Upsilon_j,    
\end{equation}
parameterized by two nonnegative sequences $\{\Gamma_\ell\mid \ell\in\natu\cup \{0\}\}$ and $\{\Upsilon_j \mid j\in\natu\}$. Here, the factor $\Gamma_{|v|}$ introduces a dependence on the order (the size $|v|$) of the subset $v$. POD weights arise naturally in the analysis of QMC methods for partial differential equations with random coefficients \cite{dick2016higher,graham:2015,kuo2017multilevel,kuo2012quasi,kuo2015multi}. See \cite{kuo:nuye:2016} for a comprehensive overview.

For the specific yet important class of POD weights where $\Gamma_{|v|}=(|v|!)^\sigma$ for $\sigma\geq 0$, a common starting point is the inequality
\begin{equation}\label{eqn:naivebound}
\sum_{v\subseteq \natu} |v|! m^{|v|}\prod_{j\in v} \Upsilon_j\leq \sum_{\ell=0}^\infty \Big(m\sum_{j=1}^\infty \Upsilon_j \Big)^{\ell}   
\end{equation}
established in \cite[Lemma 6.3]{kuo2012quasi}. This bound directly implies the finiteness of $S_{\bsgamma}(m)$ for the case $\sigma=1$, provided $m\sum_{j=1}^\infty \Upsilon_j <1$. Handling the general case $\sigma\neq 1$ then requires additional steps, such as an application of Jensen's or Hölder's inequality, as illustrated in the proof of \cite[Theorem 6.4]{kuo2012quasi}.

While useful, inequality \eqref{eqn:naivebound} yields a conservative estimate. It suggests the sum could diverge when $m$ exceeds $(\sum_{j=1}^\infty \Upsilon_j)^{-1}$, which severely overestimates its actual growth. In contrast, our Theorem~\ref{thm:improve} will demonstrate that the sum remains finite for all $m>0$ under the sole condition $\sum_{j=1}^\infty \Upsilon_j<\infty$.

The proof of this stronger convergence result is powered by the following general bounding theorem, which serves as our core analytical engine:
\begin{theorem}\label{mainthm}
Let $S_{\bsgamma}(m)=\sum_{v\subseteq \natu}\gamma_v m^{|v|}$ where $\gamma_v$ are of POD form \eqref{eqn:PODweight} with $\sum_{j=1}^\infty \Upsilon_j<\infty$. Then
$$S_{\bsgamma}(m)\leq \Gamma_0+\sum_{\ell=1}^\infty \exp(\ell+1) \Gamma_\ell m^\ell  \prod_{J=1}^\ell \max\left(\Upsilon_J,\frac{1}{J}\sum_{j=J+1}^\infty \Upsilon_j\right).$$
\end{theorem}

The remainder of this paper is organized as follows. The proof of Theorem~\ref{mainthm} is presented in Section~\ref{sec:proof}. Section~\ref{sec:corollary} explores its important corollaries, including precise asymptotic rates for a common class of weights $\gamma_v=(|v|!)^{\sigma}\prod_{j\in v}j^{-\rho}$. Section~\ref{sec:SPOD} extends our findings to the broader class of smoothness-driven product and order dependent (SPOD) weights. Section~\ref{sec:QMC} applies our theoretical results to QMC integration, where we show that a common assumption in the QMC literature can be removed while retaining the same convergence rate.

\section{Proof of Theorem~\ref{mainthm}}\label{sec:proof}

Our proof requires the following technical lemma:

\begin{lemma}\label{lem:factorial}
Let $\natu_{\ell}=\{v\subseteq \natu\mid |v|=\ell\}$ for $\ell\in\natu$. Given nonnegative $\{\Upsilon_j\mid j\in \natu\}$ satisfying $\sum_{j=1}^\infty \Upsilon_j<\infty$, we denote
$\zeta_{J}=\sum_{j=J}^\infty \Upsilon_j$.
Then
    $$\sum_{v\in \natu_\ell} \prod_{j\in v} \Upsilon_j\leq \prod_{J=1}^\ell \left(\Upsilon_J+\frac{\zeta_{J+1}}{\ell-J+1}\right)\leq \exp(\ell+1)\prod_{J=1}^\ell \max(\Upsilon_J,J^{-1}\zeta_{J+1}). $$
\end{lemma}
\begin{proof}
Without loss of generality, we may assume that $\Upsilon_j>0$ for every $j\in\natu$. Indeed, if $\Upsilon_j=0$ for some $j$, any subset $v$ containing such a $j$ has $\prod_{j\in v}\Upsilon_j=0$. Hence, we may restrict the summation over $\natu_\ell$ to subsets consisting only of indices with positive $\Upsilon_j$.

    Let $\mathbb{P}$ be the probability distribution on $\natu$ defined by the probability mass function $\mathbb{P}({j})  =\Upsilon_j/\zeta_1$ for $j\in \natu$. Consider $X_1,X_2,\dots,X_\ell\overset{\text{i.i.d.}}{\sim} \mathbb{P}$. The probability that $X_1,X_2,\dots,X_\ell$ are all distinct is given by
    \begin{equation}\label{eqn:prob}
     \sum_{v\in \natu_\ell} \ell!\prod_{j\in v}\frac{\Upsilon_j}{\zeta_1}=  \frac{\ell!}{\zeta_1^\ell}\sum_{v\in \natu_\ell} \prod_{j\in v}\Upsilon_j.   
    \end{equation}
    Here, the factor $\ell!$ accounts for all possible orderings of the $\ell$ distinct indices that form the set $v$.

    Next, we let $I_j=\sum_{i=1}^\ell \bsone\{X_i=j\}$, where $\bsone$ denotes the indicator function, and $\mathcal{I}_j$ be the $\sigma-$algebra generated by $I_1,\dots,I_j$. The event $X_1,X_2,\dots,X_\ell$ are all distinct is equivalent to $I_j\leq 1$ for all $j\in\natu$. Denoting $A_j=\bsone\{I_j\leq 1\}$, for each $J\in \natu$ with $J\geq 2$ we have
    \begin{align*}
    \e\Big(\prod_{j=1}^{J} A_j\Big\vert \mathcal{I}_{J-1}\Big)=\Big(\prod_{j=1}^{J-1} A_j\Big) \e\Big( A_{J}\Big\vert \mathcal{I}_{J-1}\Big)   
    =\Big(\prod_{j=1}^{J-1} A_j\Big)\Pr\Big(I_{J}\leq 1 \Big\vert \mathcal{I}_{J-1}\Big).
    \end{align*}
    Conditioned on $\mathcal{I}_{J-1}$, $I_{J}$ follows a binomial distribution with a total number of trials $\ell-\sum_{j=1}^{J-1}I_j$ and success probability $\Upsilon_J/\zeta_J$. When $2\leq J\leq \ell$, 
    because $\sum_{j=1}^{J-1}I_j\leq J-1$ given $I_j\leq 1$ for all $1\leq j\leq J-1$, 
    \begin{align*}
   &\Big(\prod_{j=1}^{J-1} A_j\Big)\Pr\Big(I_{J}\leq 1 \Big\vert \mathcal{I}_{J-1}\Big)\\
   \leq &\Big(\prod_{j=1}^{J-1} A_j\Big)\left(\Big(1-\frac{\Upsilon_J}{\zeta_J}\Big)^{\ell-J+1}+(\ell-J+1)\frac{\Upsilon_J}{\zeta_J}\Big(1-\frac{\Upsilon_J}{\zeta_J}\Big)^{\ell-J}\right)  \\
   =& \Big(\prod_{j=1}^{J-1} A_j\Big)\Big(1+(\ell-J)\frac{\Upsilon_J}{\zeta_J}\Big)\Big(1-\frac{\Upsilon_J}{\zeta_J}\Big)^{\ell-J}. 
    \end{align*}
    Taking expectation on both sides, we obtain
    \begin{align*}
        \e\Big(\prod_{j=1}^{J} A_j\Big)
       \leq\e\Big(\prod_{j=1}^{J-1} A_j\Big)\Big(1+(\ell-J)\frac{\Upsilon_J}{\zeta_J}\Big)\Big(1-\frac{\Upsilon_J}{\zeta_J}\Big)^{\ell-J}.
    \end{align*}
    By induction and the fact that $\prod_{j=1}^{\infty} A_j=1$ only if $\prod_{j=1}^{\ell} A_j=1$, 
    \begin{align*}
     \e\Big(\prod_{j=1}^{\infty} A_j\Big)
     \leq \e\Big(\prod_{j=1}^{\ell} A_j\Big)\leq &\e(A_1) \prod_{J=2}^\ell \Big(1+(\ell-J)\frac{\Upsilon_J}{\zeta_J}\Big)\Big(1-\frac{\Upsilon_J}{\zeta_J}\Big)^{\ell-J} \\
     =& \prod_{J=1}^\ell \Big(1+(\ell-J)\frac{\Upsilon_J}{\zeta_J}\Big)\Big(1-\frac{\Upsilon_J}{\zeta_J}\Big)^{\ell-J},
    \end{align*}
    where the last equality follows because $\e(A_1)=\Pr(I_1\leq1)$ and $I_1$ follows a binomial distribution with a total number of trials $\ell$ and success probability $\Upsilon_1/\zeta_1$. Equating $\e(\prod_{j=1}^{\infty} A_j)$ with \eqref{eqn:prob} yields the upper bound
    \begin{equation*}
     \sum_{v\in \natu_\ell} \prod_{j\in v}\Upsilon_j\leq \frac{\zeta_1^\ell}{\ell!}\prod_{J=1}^\ell \Big(1+(\ell-J)\frac{\Upsilon_J}{\zeta_J}\Big)\Big(1-\frac{\Upsilon_J}{\zeta_J}\Big)^{\ell-J}.   
    \end{equation*}
Furthermore, because $\zeta_J=\Upsilon_J+\zeta_{J+1}$,
\begin{align*}
    \zeta_1^\ell\prod_{J=1}^\ell\Big(1-\frac{\Upsilon_J}{\zeta_J}\Big)^{\ell-J}=\zeta_1^\ell\prod_{J=1}^\ell\Big(\frac{\zeta_{J+1}}{\zeta_{J}}\Big)^{\ell-J}=\prod_{J=1}^\ell \zeta_J.
\end{align*}
Therefore,
\begin{align*}
    \sum_{v\in \natu_\ell} \prod_{j\in v}\Upsilon_j \leq & \frac{1}{\ell!}\prod_{J=1}^\ell \Big(\zeta_J+(\ell-J)\Upsilon_J\Big)\\
    =&\prod_{J=1}^\ell \left(\frac{\zeta_{J+1}}{\ell-J+1}+\Upsilon_J\right)\\
    \leq & \prod_{J=1}^\ell \left(\max(J^{-1}\zeta_{J+1},\Upsilon_J) \left(\frac{J}{\ell-J+1}+1\right)\right).
\end{align*}
Finally, we use the Stirling's formula \cite{robbins1955remark} to conclude
\begin{align*}
        \prod_{J=1}^\ell \left(\frac{J}{\ell-J+1}+1\right)
        =&\prod_{J=1}^\ell \frac{\ell+1}{\ell-J+1}=\frac{(\ell+1)^{\ell+1}}{(\ell+1)!}\leq \exp(\ell+1).\qedhere
    \end{align*}
\end{proof}

\begin{proof}[Proof of Theorem~\ref{mainthm}]
    Notice that
    $$S_{\bsgamma}(m)=\Gamma_0+\sum_{\ell=1}^\infty\sum_{v\in \natu_\ell }\gamma_v m^\ell =\Gamma_0+\sum_{\ell=1}^\infty \Gamma_\ell m^{\ell}\sum_{v\in \natu_\ell }\prod_{j\in v} \Upsilon_j. $$
    Our conclusion follows directly from Lemma~\ref{lem:factorial}.
\end{proof}

\section{Tractability analysis of POD weights}\label{sec:corollary}

In this section, we apply Theorem~\ref{mainthm} to derive improved bounds for several common settings in QMC tractability analysis. We first present the following result, which sharpens the key estimate in \cite[Lemma 6.3]{kuo2012quasi}.

\begin{theorem}\label{thm:improve}
For nonnegative $\{\Upsilon_j\mid j\in \natu\}$, 
\begin{equation}\label{eqn:improve}
   \sum_{v\subseteq \natu} |v|!  m^{|v|}\prod_{j\in v} \Upsilon_j<\infty \text{ for all } m>0 
\end{equation}
    if and only if $\sum_{j=1}^\infty \Upsilon_j<\infty$.
\end{theorem}
\begin{proof}
First, we assume $\sum_{j=1}^\infty \Upsilon_j<\infty$.
    A direct application of Theorem~\ref{mainthm} with $\Gamma_\ell=\ell!$ yields
    \begin{equation}\label{eqn:improvebound}
    \sum_{v\subseteq \natu} |v|!  m^{|v|}\prod_{j\in v} \Upsilon_j\leq  \Gamma_0+\sum_{\ell=1}^\infty \exp(\ell+1) m^\ell \prod_{J=1}^\ell \max\left(J\Upsilon_J,\sum_{j=J+1}^\infty \Upsilon_j\right).    
    \end{equation}
By the inequality of arithmetic and geometric means, 
$$\left(\prod_{J=1}^\ell \max\left(J\Upsilon_J,\sum_{j=J+1}^\infty \Upsilon_j\right)\right)^{1/\ell}\leq \frac{1}{\ell}\sum_{J=1}^\ell J\Upsilon_J+\frac{1}{\ell}\sum_{J=1}^\ell\sum_{j=J+1}^\infty \Upsilon_j.$$
       Since  $\sum_{j=1}^\infty \Upsilon_j<\infty$, for every $\varepsilon>0$ we can find a sufficiently large $J^*$ such that $\sum_{j=J^*+1}^\infty \Upsilon_j<\varepsilon$. Then for every $\ell> J^*$,
       \begin{align*}
 \frac{1}{\ell}\sum_{J=1}^\ell J\Upsilon_J+\frac{1}{\ell}\sum_{J=1}^\ell\sum_{j=J+1}^\infty \Upsilon_j\leq  2\varepsilon+  \frac{1}{\ell}\sum_{J=1}^{J^*}J\Upsilon_J+\frac{1}{\ell}\sum_{J=1}^{J^*}\sum_{j=J+1}^\infty \Upsilon_j.
       \end{align*}
       Taking the limit $\ell\to\infty$ yields
       $$\limsup_{\ell\to\infty}\left(\prod_{J=1}^\ell \max\left(J\Upsilon_J,\sum_{j=J+1}^\infty \Upsilon_j\right)\right)^{1/\ell}\leq 2\varepsilon.$$
   Since $\varepsilon>0$ can be arbitrarily small, the left hand side in fact equals to $0$. Consequently, \eqref{eqn:improve} follows from \eqref{eqn:improvebound} and the root test.

   Next, we assume $\sum_{j=1}^\infty \Upsilon_j=\infty$. We have the lower bound
   $$\sum_{v\subseteq \natu} |v|! m^{|v|}\prod_{j\in v} \Upsilon_j\geq \sum_{v\subseteq \natu}m^{|v|}\prod_{j\in v} \Upsilon_j=\prod_{j=1}^\infty (1+m\Upsilon_j),$$
   which diverges given $\sum_{j=1}^\infty \Upsilon_j=\infty$ \cite[Proposition 5.4]{conway1978functions}.
\end{proof}

\begin{remark}
    An alternative way to prove the sufficiency of $\sum_{j=1}^\infty \Upsilon_j<\infty$ is via the Gamma integral
    $$|v|!=\int_{0}^\infty t^{|v|} \exp(-t)\rd t,$$
    which allows us to rewrite the sum as
   \begin{align*}
       \sum_{v\subseteq \natu} |v|!  m^{|v|}\prod_{j\in v} \Upsilon_j=&\int_{0}^\infty \sum_{v\subseteq \natu}t^{|v|} \exp(-t)m^{|v|}\prod_{j\in v} \Upsilon_j\rd t\\
       =&\int_{0}^\infty \exp(-t) \prod_{j=1}^{\infty}\left(1+tm\Upsilon_j\right)\rd t.
   \end{align*}
   The interchange of integration and summation is valid since all integrands are nonnegative. The proof is then completed by applying \cite[Lemma 3]{HICKERNELL2003286}, which states that the infinite product $\prod_{j=1}^{\infty}\left(1+tm\Upsilon_j\right)$ is sub-exponential in $t$ when $\sum_{j=1}^\infty \Upsilon_j<\infty$.
\end{remark}

Next, we study the case $\Gamma_{|v|}$ is a power of $|v|!$ and $\Upsilon_j$ is proportional $j^{-\rho}$ for $\rho>1$. We need the following lemma:

\begin{lemma}\label{lem:growth}
    For $\theta>0$,
    \begin{equation*}
\lim_{m\to\infty}m^{-1/\theta}\log\left(\sum_{\ell=0}^\infty\frac{m^\ell}{(\ell!)^{\theta}}  \right)=\theta.
    \end{equation*}
\end{lemma}
\begin{proof}
    Let $\ell_*=\lfloor m^{1/\theta}\rfloor$. Because 
    $$\frac{\left((\ell+1)!\right)^{-\theta}m^{\ell+1}}{(\ell!)^{-\theta}m^\ell}=(\ell+1)^{-\theta}m < 1$$
    if and only if $\ell\geq \ell_*$, we see that $(\ell!)^{-\theta}m^\ell$ is maximized at $\ell=\ell_*$. Therefore,
    $$\frac{m^{\ell_*}}{(\ell_*!)^{\theta}}\leq \sum_{\ell=0}^\infty\frac{m^\ell}{(\ell!)^{\theta}}  \leq 1+2\ell_*\frac{m^{\ell_*}}{(\ell_*!)^{\theta}}+ \sum_{\ell=2\ell_*+1}^\infty \frac{m^\ell}{(\ell!)^{\theta}}  \leq 1+\left( 2\ell_*+\sum_{\ell=0}^\infty 2^{-\theta\ell}\right)\frac{m^{\ell_*}}{(\ell_*!)^{\theta}},$$
    where we have used $(\ell+1)^{-\theta}m<2^{-\theta}$ when $\ell\ge 2\ell_*+1$. By the Stirling's formula, $\log((\ell_*!)^{-\theta}m^{\ell_*})\sim \theta m^{1/\theta}$. Our conclusion then follows because the lower and upper bounds converge to the same limit.
\end{proof}

\begin{theorem}\label{thm:growth}
Given $\rho>1$, $\sigma\geq 0$ and $C_{\Upsilon}>0$, let $S_{\bsgamma}(m)=\sum_{v\subseteq \natu}\gamma_v m^{|v|}$ with $\gamma_\emptyset=1$ and
    $$\gamma_v=(|v|!)^{\sigma}\prod_{j\in v} \frac{C_{\Upsilon}}{j^{\rho}} \quad\text{for}\quad v\neq\emptyset. $$
    Then $S_{\bsgamma}(m)<\infty$ for all $m> 0$ if and only if  $\rho>\sigma$, in which case
    $$\liminf_{m\to\infty} m^{-\frac{1}{\rho-\sigma}}\log S_{\bsgamma}(m)\geq (\rho-\sigma)C^{\frac{1}{\rho-\sigma}}_{\Upsilon}$$
    and
    $$\limsup_{m\to\infty}m^{-\frac{1}{\rho-\sigma}}\log S_{\bsgamma}(m)\leq (\rho-\sigma)(C_{\rho}C_\Upsilon)^{\frac{1}{\rho-\sigma}}$$
    for $C_{\rho}=\exp(1)/\min(\rho-1,1)$.
\end{theorem}
\begin{proof}
   Let $\Gamma_\ell=(\ell!)^{\sigma}$ and $\Upsilon_j=C_\Upsilon j^{-\rho}$ in Theorem~\ref{mainthm}. Because
   $$\frac{1}{J}\sum_{j=J+1}^\infty \Upsilon_j\leq \frac{C_{\Upsilon}}{J} \int_{J}^\infty j^{-\rho} \rd j=\frac{C_{\Upsilon}J^{-\rho}}{\rho-1},$$
   we have the upper bound
$$S_{\bsgamma}(m)\leq 1+\sum_{\ell=1}^\infty (\ell!)^\sigma \frac{\exp(\ell+1)(C_{\Upsilon} m)^\ell}{\min(\rho-1,1)^\ell} \prod_{J=1}^\ell J^{-\rho}\leq \exp(1) \sum_{\ell=0}^\infty\frac{(C_\rho C_{\Upsilon} m)^\ell}{(\ell!)^{\rho-\sigma}}.  $$
Meanwhile, we have the lower bound
    $$S_{\bsgamma}(m)\geq 1+\sum_{\ell=1 }^\infty (\ell !)^{\sigma}m^{\ell}\prod_{j=1}^\ell \frac{C_\Upsilon}{j^{\rho}}=  \sum_{\ell=0}^\infty \frac{(C_\Upsilon m)^{\ell}}{(\ell!)^{\rho-\sigma}} .$$
When $\rho\leq \sigma$, the lower bound diverges when $m\geq C^{-1}_\Upsilon $. When $\rho>\sigma$, applying Lemma~\ref{lem:growth} to both the upper and lower bounds yields our conclusion.
\end{proof}

\begin{remark}
   In the context of tractability analysis, Theorem~\ref{thm:growth} implies that relaxing $\gamma_v=(|v|!)^{\sigma}\prod_{j\in v}C_{\Upsilon}j^{-\rho}$ to $\gamma_v\leq \prod_{j\in v}C_{\Upsilon}j^{-\rho+\sigma}$ does not change the growth of $\log S_{\bsgamma}(m)$ up to a constant factor. This relaxation is particularly useful because the latter product-form upper bounds are easier to analyze and implement numerically. For instance, a recent study \cite{pan2026uncertainty} employs this relaxation to design an importance sampling measure for QMC integration.
\end{remark}

\section{Extension to SPOD weights}\label{sec:SPOD}

In this section, we extend our main results to the broader framework of smoothness-driven product and order dependent (SPOD) weights introduced in \cite{dick2014higher}. For a fixed smoothness parameter $\alpha\in\natu$, we denote $\{1{:}\alpha\}=\{1,2,\dots,\alpha\}$. Following \cite{cools2021fast,dick2016multilevel}, we define these weights for nonempty $v\subseteq \natu$ by
\begin{equation}\label{eqn:SPODdef}
\gamma_v=\sum_{\bsnu\in \{1{:}\alpha\}^{|v|}}\Gamma_{|\bsnu|}\prod_{j\in v} \Upsilon_{j,\nu_j},  
\end{equation}
where $|\bsnu|=\sum_{j\in v}\nu_j$ for a multi-index $\bsnu=(\nu_j)_{j\in v}$, and the parameters are nonnegative sequences $\{\Gamma_\ell\mid \ell\in\natu\cup \{0\}\}$ and $\{\Upsilon_{j,k} \mid j\in\natu,k\in \{1{:}\alpha\}\}$.  This general form reduces to the standard POD weights \eqref{eqn:PODweight} when $\alpha=1$.

First, we establish the counterpart of Theorem~\ref{thm:improve}.
\begin{theorem}\label{thm:SPODimprove}
For $\alpha\in\natu$ and nonnegative $\{\Upsilon_{j,k} \mid j\in\natu,k\in \{1{:}\alpha\}\}$, 
\begin{equation}\label{eqn:SPODsum}
   \sum_{v\subseteq \natu}  m^{|v|}\sum_{\bsnu\in \{1{:}\alpha\}^{|v|}}|\bsnu|!\prod_{j\in v} \Upsilon_{j,\nu_j}<\infty \text{ for all } m>0
\end{equation}
    if $\sum_{j=1}^\infty\sum_{k=1}^\alpha \Upsilon^{1/k}_{j,k}<\infty$.
\end{theorem}
\begin{proof}
It suffices to prove the statement for $m\geq 1$. Let 
    $$\Upsilon'_{j'}=\max_{k\in \{1{:}\alpha\}}\Upsilon^{1/k}_{j,k} \quad \text{for} \quad  \alpha (j-1)+1\leq j' \leq \alpha j,$$
    and
    $$S_{\bsgamma}(m)=\sum_{v'\subseteq\natu} |v'|! m^{|v'|}\prod_{j'\in v'}\Upsilon'_{j'}.$$
    For each $(v,\bsnu)$ pair appearing in \eqref{eqn:SPODsum}, 
    \begin{equation}\label{eqn:SPODtoPOD}
   |\bsnu|!  m^{|v|}\prod_{j\in v} \Upsilon_{j,\nu_j}\leq |v'|! m^{|v'|}\prod_{j'\in v'}\Upsilon'_{j'}   
    \end{equation}
     with $v'=\{\alpha(j-1)+k\mid j\in v,k\in \{1{:}\nu_j\}\}$. Therefore, the left hand side of \eqref{eqn:SPODsum} is bounded by $S_{\bsgamma}(m)$ and its finiteness follows from Theorem~\ref{thm:improve} and 
    \begin{align*}
     \sum_{j'=1}^\infty\Upsilon'_{j'}&=\alpha \sum_{j=1}^\infty\max_{k\in \{1{:}\alpha\}}\Upsilon^{1/k}_{j,k} \leq \alpha \sum_{j=1}^\infty\sum_{k=1}^\alpha \Upsilon^{1/k}_{j,k}<\infty. \qedhere  
    \end{align*}
\end{proof}
\begin{remark}

The converse of Theorem~\ref{thm:SPODimprove} is unfortunately not true. For instance, let $\alpha=2$, $\Upsilon_{j,1}=0$, and $\Upsilon_{j,2}=j^{-2}(1+\log j)^{-2}$ for $j\in \natu$. Then $\sum_{j=1}^\infty\sum_{k=1}^\alpha \Upsilon^{1/k}_{j,k}=\sum_{j=1}^\infty j^{-1}(1+\log j)^{-1}=\infty$. However,
\begin{align*}
    \sum_{v\subseteq \natu}  m^{|v|}\sum_{\bsnu\in \{1{:}\alpha\}^{|v|}}|\bsnu|!\prod_{j\in v} \Upsilon_{j,\nu_j}
= &\sum_{v\subseteq \natu}  m^{|v|}(2|v|)!\prod_{j\in v} j^{-2}(1+\log j)^{-2}\nonumber \\
\overset{\text{(i)}}{\leq} & \sum_{v\subseteq \natu} (4m)^{|v|}|v|! \prod_{j\in v} j^{-1}(1+\log j)^{-2}\overset{\text{(ii)}}{<}\infty,
\end{align*}
where (i) is due to
$$(2|v|)!=(|v|!)^2{2|v|\choose |v|}\le (|v|!)^2 2^{2|v|}\leq 4^{|v|} |v|! \prod_{j\in v} j,$$
and (ii) follows from Theorem~\ref{thm:improve} and $\sum_{j=1}^\infty j^{-1}(1+\log j)^{-2}<\infty$.
\end{remark}

\begin{remark}
    In the special case $\Upsilon_{j,k}=\Upsilon^k_{j,1}$ for all $k\in\{1{:}\alpha\}$, the sufficient condition $\sum_{j=1}^\infty\sum_{k=1}^\alpha \Upsilon^{1/k}_{j,k}<\infty$ is equivalent to $\sum_{j=1}^\infty\Upsilon_{j,1}<\infty$, so by Theorem~\ref{thm:improve}, this condition is also necessary.
\end{remark}

The next theorem generalizes Theorem~\ref{thm:growth}.

\begin{theorem}
Given $\alpha\in\natu$, $\rho>1$, $\sigma\geq 0$ and positive $\{C_{\Upsilon,k}\mid k\in \{1{:}\alpha\}\}$, let $S_{\bsgamma}(m)=\sum_{v\subseteq \natu}\gamma_v m^{|v|}$ with $\gamma_\emptyset=1$ and
    $$\gamma_v=\sum_{\bsnu\in \{1{:}\alpha\}^{|v|}}(|\bsnu|!)^{\sigma}\prod_{j\in v} \frac{C_{\Upsilon,\nu_j}}{j^{\nu_j\rho}} \quad\text{for}\quad v\neq \emptyset. $$
    Then $S_{\bsgamma}(m)<\infty$ for all $m> 0$ if and only if  $\rho>\sigma$, in which case
    \begin{equation}\label{eqn:liminf}
    \liminf_{m\to\infty} m^{-\frac{1}{\rho-\sigma}}\log S_{\bsgamma}(m)\geq (\rho-\sigma)C^{\frac{1}{\rho-\sigma}}_{\Upsilon,1}    
    \end{equation}
    and
    \begin{equation}\label{eqn:limsup}
     \limsup_{m\to\infty}m^{-\frac{1}{\rho-\sigma}}\log S_{\bsgamma}(m)\leq (\rho-\sigma)\left(C_{\alpha,\rho}\max_{k\in \{1{:}\alpha\}}C^{1/k}_{\Upsilon,k} \right)^{\frac{1}{\rho-\sigma}}   
    \end{equation}
    for a finite constant $C_{\alpha,\rho}$ depending on $\alpha$ and $\rho$.
\end{theorem}
\begin{proof}
Because 
$$S_{\bsgamma}(m)\geq \sum_{v\subseteq \natu} (|v|!)^{\sigma} m^{|v|} \prod_{j\in v} \frac{C_{\Upsilon,1}}{j^{\rho}},$$
Theorem~\ref{thm:growth} shows $S_{\bsgamma}(m)$ diverges for sufficiently large $m$ if $\rho\leq \sigma$ and \eqref{eqn:liminf} holds if $\rho>\sigma$.

Now we assume $\rho>\sigma$ and establish an upper bound on $S_{\bsgamma}(m)$. Let 
    $$\Sigma_{\ell,\ell'}=\left\{(v,\bsnu)\mid v\subseteq \natu, |v|=\ell,\bsnu\in \{1{:}\alpha\}^{|v|},|\bsnu|=\ell'\right\}.$$
    Then we can rewrite $S_{\bsgamma}(m)$ as
    \begin{equation}\label{eqn:SPODSm}
    S_{\bsgamma}(m)=1+\sum_{\ell=1}^\infty m^{\ell}\sum_{\ell'=\ell}^{\alpha \ell}    \left(\ell'!\right)^{\sigma}\sum_{(v,\bsnu)\in \Sigma_{\ell,\ell'}} \prod_{j\in v} \frac{C_{\Upsilon,\nu_j}}{j^{\nu_j\rho}}.    
    \end{equation}
    Denote $C_{\max}=\max_{k\in \{1{:}\alpha\}}C^{1/k}_{\Upsilon,k} $ and $\Upsilon'_{j'}=C_{\max} (\lceil j'/\alpha\rceil)^{-\rho} $ for $j'\in\natu$.
    By an argument similar to \eqref{eqn:SPODtoPOD}, Lemma~\ref{lem:factorial} implies
    \begin{align*}
        \sum_{(v,\bsnu)\in \Sigma_{\ell,\ell'}} \prod_{j\in v} \frac{C_{\Upsilon,\nu_j}}{j^{\nu_j\rho}} \leq \sum_{v'\in \natu_{\ell'}}\prod_{j'\in v'} \Upsilon'_{j'}\leq \exp(\ell'+1)\prod_{J=1}^{\ell'} \max\left(\Upsilon'_J,\frac{1}{J}\sum_{j'=J+1}^\infty \Upsilon'_{j'}\right).
    \end{align*} 
    Because $\Upsilon'_{J}\leq C_{\max}(J/\alpha)^{-\rho}$ and 
    $$\frac{1}{J}\sum_{j'=J+1}^\infty \Upsilon'_{j'}\leq \frac{\alpha C_{\max}}{J}\sum_{j=\lfloor J/\alpha\rfloor+1}j^{-\rho}\leq C'_{\alpha,\rho} C_{\max} (J/\alpha)^{-\rho} $$
    for a constant $C'_{\alpha,\rho}\geq 1$ depending on $\alpha$ and $\rho$,
    we can further bound
    \begin{align*}
       \sum_{\ell'=\ell}^{\alpha \ell}    \left(\ell'!\right)^{\sigma}\sum_{(v,\bsnu)\in \Sigma_{\ell,\ell'}} \prod_{j\in v} \frac{C_{\Upsilon,\nu_j}}{j^{\nu_j\rho}}
       \leq &  \sum_{\ell'=\ell}^{\alpha\ell}    \left(\ell'!\right)^{\sigma} \exp(\ell'+1)(\alpha^\rho C'_{\alpha,\rho}C_{\max})^{\ell'} \prod_{J=1}^{\ell'} J^{-\rho}\\
       = &\exp(1) \sum_{\ell'=\ell}^{\alpha\ell} \frac{\left(C_{\alpha,\rho}C_{\max}\right)^{\ell'} }{(\ell'!)^{\rho-\sigma}},
    \end{align*}
    where  $C_{\alpha,\rho}=\exp(1) \alpha^\rho C'_{\alpha,\rho}$. Let $\ell_*= \lceil (2C_{\alpha,\rho}C_{\max})^{1/(\rho-\sigma)}\rceil$. Then $\ell>\ell_*$ implies
     $(\ell')^{\rho-\sigma}\geq 2C_{\alpha,\rho}C_{\max}$ for $\ell'\geq\ell$ and
    $$\sum_{\ell'=\ell}^{\alpha\ell} \frac{\left( C_{\alpha,\rho}C_{\max}\right)^{\ell'} }{(\ell'!)^{\rho-\sigma}}\leq \frac{\left(C_{\alpha,\rho}C_{\max}\right)^{\ell} }{(\ell!)^{\rho-\sigma}}\sum_{\ell'=0}^\infty 2^{-\ell'}
    = \frac{2\left(C_{\alpha,\rho}C_{\max}\right)^{\ell} }{(\ell!)^{\rho-\sigma}}.$$
    Plugging the above bounds into \eqref{eqn:SPODSm} yields
    \begin{equation}\label{eqn:SmSml}
     S_{\bsgamma}(m)\leq S_{\bsgamma}(m,\ell_*)
    +2\exp(1)\sum_{\ell=\ell_*+1}^\infty  \frac{\left(C_{\alpha,\rho}C_{\max}m\right)^{\ell} }{(\ell!)^{\rho-\sigma}},    
    \end{equation}
    where
    $$S_{\bsgamma}(m,\ell_*)=1+\exp(1)\sum_{\ell=1}^{\ell_*} m^{\ell}\sum_{\ell'=\ell}^{\alpha\ell} \frac{\left(C_{\alpha,\rho}C_{\max}\right)^{\ell'} }{(\ell'!)^{\rho-\sigma}}.$$
    Because $S_{\bsgamma}(m,\ell_*)$ grows polynomially in $m$, \eqref{eqn:limsup} follows after we apply Lemma~\ref{lem:growth} to the second term on the right hand side of \eqref{eqn:SmSml}.
\end{proof}

\section{Application to QMC integration}\label{sec:QMC}

We conclude this paper by applying our theory to strengthen Theorem 3.1 of \cite{dick2014higher}, which addresses QMC integration for integrands in weighted Sobolev spaces equipped with SPOD weights. The quadrature rule  employed therein is the CBC-constructed interlaced polynomial lattice rule of order $\alpha$, with $\alpha\geq 2$ an integer. Its points are constructed adaptively with respect to the weights to minimize the worst-case integration error over the unit ball in the weighted Sobolev space. Since a full exposition is beyond our present scope, we refer the reader to \cite{dick2014higher,dick:pill:2010} for the error analysis and implementation details.

We first set up the notation required for the theorem. For $\alpha\in\natu$ and a positive sequence $\bsbeta=\{\beta_j\mid j\in \natu\}$, we introduce the SPOD weights
\begin{equation}\label{eqn:betaSPOD}
    \gamma_{v,\alpha,\bsbeta}=\sum_{\bsnu\in \{1{:}\alpha\}^{|v|}}\gamma_{\bsnu,\bsbeta}\quad\text{for}\quad \gamma_{\bsnu,\bsbeta}=|\bsnu|!\prod_{j\in v} \beta_j^{\nu_j}.
\end{equation}
Next, for $\alpha\in \natu$ and a smooth integrand $F:[0,1]^s\to \real$, we use $\partial^{\bsnu} F$ with $\bsnu=(\nu_j)_{j\in v}\in \{1{:}\alpha\}^{|v|}$ to denote the partial derivative
$$\partial^{\bsnu} F(\bsx)=\left(\prod_{j\in v} \frac{\partial^{\nu_j}}{\partial x_j^{\nu_j}}\right)F(\bsx).$$
We further denote its supremum norm by
$$\Vert\partial^{\bsnu} F\Vert_{\infty}=\sup_{\bsx\in [0,1]^s} |\partial^{\bsnu} F(\bsx)|.$$
If $\Vert\partial^{\bsnu} F\Vert_{\infty}<\infty$ for all $\bsnu\in \{1{:}\alpha\}^{|v|}$ and all $v\subseteq \{1{:}s\}$, we define the weighted norm of $F$ as
$$\Vert F\Vert_{s,\alpha,\bsbeta}=\sup_{v\subseteq \{1{:}s\}} \sup_{\bsnu\in \{1{:}\alpha\}^{|v|}}\gamma_{\bsnu,\bsbeta}^{-1}\Vert\partial^{\bsnu} F\Vert_{\infty},$$
where $\gamma_{\bsnu,\bsbeta}$ is defined as in \eqref{eqn:betaSPOD} when $v\neq\emptyset$, and $\gamma_{\bsnu,\bsbeta}=1$ when $v=\emptyset$.

The following theorem improves \cite[Theorem 3.1]{dick2014higher} by removing an extra assumption that was required there. Specifically, we prove that the integration error achieves the same dimension-independent convergence rate without the condition that $\sum_{j=1}^\infty\beta_j$ be less than a finite threshold, which had been imposed only in the case where $\sum_{j=1}^\infty\beta^p_j$ diverges for all $p\in (0,1]$ except for $p=1$. We note that this condition is imposed in other QMC works (see, e.g., \cite{dick2016higher,graham:2015,kuo2012quasi,kuo2015multi}), and our analysis below generalizes to these settings as well.

\begin{theorem}
    Let $\bsbeta=\{\beta_j\mid j\in \natu\}$ be a positive sequence with $\sum_{j=1}^\infty\beta^p_j<\infty$ for some $p\in (0,1]$. Define $\alpha=\lfloor 1/p\rfloor+1$. For $s\in\natu$ and $N$ an integer power of a prime number $b$, let $\{\bsx_0,\bsx_1,\dots,\bsx_{N-1}\}\subseteq [0,1]^s$ be the quadrature points of the interlaced polynomial lattice rule of order $\alpha$, constructed with respect to the SPOD weights \eqref{eqn:betaSPOD} (see \cite[Section 3]{dick2014higher} for a full description).

    Then for any integrand $F:[0,1]^s\to \real$ with $\Vert F\Vert_{s,\alpha,\bsbeta}\leq 1$, there exists a constant $C_{\alpha,\bsbeta,b,p}<\infty$ independent of $s$ such that 
    $$\left|\frac{1}{N}\sum_{n=0}^{N-1} F(\bsx_n)-\int_{[0,1]^s}F(\bsx)\rd \bsx\right|\leq C_{\alpha,\bsbeta,b,p} N^{-1/p}.$$
\end{theorem}

\begin{proof}
    The case $p\in (0,1)$ is covered by the original theorem \cite[Theorem 3.1]{dick2014higher}, so we consider only the case $p=1$. Our proof relies on \cite[Theorem 3.10]{dick2014higher}, which states that for $\lambda\in (1/\alpha,1]$ and $F$ satisfying $\Vert F\Vert_{s,\alpha,\bsgamma,q,\infty}\leq 1$, 
    $$\left|\frac{1}{N}\sum_{n=0}^{N-1} F(\bsx_n)-\int_{[0,1]^s}F(\bsx)\rd \bsx\right|\leq \left(\frac{2}{N-1}\sum_{\emptyset\neq v\subseteq \{1{:}s\}}\gamma^{\lambda}_v \left(\rho_{\alpha,b}(\lambda)\right)^{|v|}\right)^{1/\lambda}.$$
    Here, $\bsgamma=\{\gamma_v\mid v\subseteq \natu\}$ are the weights used in the construction of the interlaced polynomial lattice rule, $\rho_{\alpha,b}(\lambda)$ is a constant depending on $\alpha,b,\lambda$ (see \cite[equation (3.37)]{dick2014higher}), and $\Vert F\Vert_{s,\alpha,\bsgamma,q,\infty}$ with $q\in [1,\infty]$ is a weighted Sobolev norm given by \cite[equation (3.7)]{dick2014higher}. While we omit the exact form of $\Vert F\Vert_{s,\alpha,\bsgamma,q,\infty}$ here, it was shown in \cite[Section 3.1]{dick2014higher} that, if $c>0$ satisfies
    \begin{equation*}
     \Vert\partial^{\bsnu} F\Vert_{\infty}\leq c \gamma_{\bsnu,\bsbeta} \text{ for all }\bsnu\in \{1{:}\alpha\}^{|v|} \text{ and all } v\subseteq \{1{:}s\},   
    \end{equation*}
    then $\Vert F\Vert_{s,\alpha,\bsgamma,q,\infty}\leq c$ for all $q\in [1,\infty]$. Since $c=\Vert F\Vert_{s,\alpha,\bsbeta}$ fulfills this condition, we have  $\Vert F\Vert_{s,\alpha,\bsgamma,q,\infty}\leq \Vert F\Vert_{s,\alpha,\bsbeta}$.
    
    In our setting, $\gamma_v=\gamma_{v,\alpha,\bsbeta}$ and $\Vert F\Vert_{s,\alpha,\bsgamma,q,\infty}\leq \Vert F\Vert_{s,\alpha,\bsbeta}\leq 1$. Thus, we can apply \cite[Theorem 3.10]{dick2014higher} with $\lambda=1$ and derive
     $$\left|\frac{1}{N}\sum_{n=0}^{N-1} F(\bsx_n)-\int_{[0,1]^s}F(\bsx)\rd \bsx\right|\leq \frac{2}{N-1}\sum_{\emptyset\neq v\subseteq \{1{:}s\}}\gamma_{v,\alpha,\bsbeta} \left(\rho_{\alpha,b}(1)\right)^{|v|}.$$
Our conclusion then follows directly from Theorem~\ref{thm:SPODimprove} with $m=\rho_{\alpha,b}(1)$ and $\Upsilon_{j,k}=\beta_j^{k}, j\in \natu,k\in \{1{:}\alpha\}$, since 
$\sum_{j=1}^\infty\sum_{k=1}^\alpha \Upsilon^{1/k}_{j,k}=\alpha \sum_{j=1}^\infty\beta_j<\infty$.
    
\end{proof}

\section*{Acknowledgments}
I thank Josef Dick and Takashi Goda for valuable discussions. 

\bibliographystyle{abbrv} 
\bibliography{qmc}

\end{document}